\documentclass{amsart}

\usepackage{subfiles}
\usepackage{amsmath,amssymb,amsthm}
\usepackage{bm}
\usepackage{mathrsfs}
\usepackage[all]{xy}
\usepackage{booktabs}
\usepackage{fullpage}
\usepackage{hyperref}
\usepackage{mathtools}
\usepackage{caption,subcaption}

\hypersetup{
 hidelinks,
 bookmarksopen=true
}

\newcommand{\rtn}{\mathbb{Q}}
\newcommand{\nnb}{\mathbb{N}}
\newcommand{\rl}{\mathbb{R}}
\newcommand{\cx}{\mathbb{C}}

\newcommand{\CL}{\mathcal{L}}

\newcommand{\CX}{\mathcal{X}}

\newcommand{\ai}{\sqrt{-1}}

\newcommand{\inj}{\hookrightarrow}

\newcommand{\isom}{\xrightarrow{\sim}}

\newcommand{\cst}{\mathrm{const}}

\newcommand{\ddbar}{\partial \bar{\partial}}
\newcommand{\tr}{\mathrm{tr}}
\newcommand{\actson}{\curvearrowright}
\newcommand{\prj}{\mathbb{P}}

\theoremstyle{plain}
\newtheorem{theorem}{Theorem}[section]
\newtheorem{lemma}[theorem]{Lemma}

\newtheorem{corollary}[theorem]{Corollary}

\theoremstyle{definition}
\newtheorem{definition}[theorem]{Definition}

\theoremstyle{definition}
\newtheorem{remark}[theorem]{Remark}

\makeindex             

\begin{document}

\title{Balanced metrics for extremal K\"ahler metrics and Fano manifolds}
\author[Y.~Hashimoto]{Yoshinori Hashimoto}
\date{\today}
\address{Department of Mathematics, Osaka Metropolitan University, 3-3-138, Sugimoto, Sumiyoshi-ku, Osaka, 558-8585, Japan.}
\email{yhashimoto@omu.ac.jp}

%

\maketitle



\begin{abstract}
	The first three sections of this paper are a survey of the author's work on balanced metrics and stability notions in algebraic geometry. The last section is devoted to proving the well-known result that a geodesically convex function on a complete Riemannian manifold admits a critical point if and only if its asymptotic slope at infinity is positive, where we present a proof which relies only on the Hopf--Rinow theorem and extends to locally compact complete length metric spaces.
\end{abstract}

\section{Introduction} \label{scitr}

Let $(X,L)$ be a pair of a compact K\"ahler manifold $X$ of complex dimension $n$ and an ample line bundle $L$, with a (fixed reference) K\"ahler metric $\omega \in c_1 (L)$. We assume $\omega$ is the K\"ahler form associated to the (fixed reference) hermitian metric $h_{\mathrm{ref}}$ on $L$, i.e.
\begin{equation*}
	\omega = \frac{1}{2 \pi \ai} \ddbar \log h_{\mathrm{ref}}.
\end{equation*}
Another hermitian metric, say $h := e^{-\phi} h_{\mathrm{ref}}$, defines another K\"ahler metric 
\begin{equation*}
	\omega_h = \omega + \ai \ddbar \phi
\end{equation*}
assuming that $\omega_h$ is positive definite; we shall also write $\omega_{\phi}$ for $\omega + \ai \ddbar \phi$.

We define $N_k := \dim_{\cx} H^0 (X, L^k)$.
	
	\begin{definition}
		Let $h$ be a hermitian metric such that $\{ s_i \}_{i=1}^{N_k}$ is a basis for $H^0 (X, L^k)$ that is orthonormal with respect to the inner product $\int_X h^k ( \cdot , \cdot  ) \omega^n_h /n!$. The \textbf{Bergman function} $\rho_k (\omega_h ) \in C^{\infty} (X, \rl)$ is defined as
		\begin{equation*}
			\rho_k (\omega_h ) := \sum_{i=1}^{N_k} |s_i|^2_{h^k} .
		\end{equation*}
	\end{definition}
	
The right hand side of the equation above depends on $h$, but it turns out that it is invariant under an overall multiplicative constant and hence depends only on the associated K\"ahler metric $\omega_h$. The Bergman function is the diagonal of the Bergman kernel. It is also called the density-of-states function or the distortion function.

The Bergman function (or more generally the Bergman kernel) has many properties as discussed in other chapters in this volume. An important result that we use in this chapter is the following.

\begin{theorem} \emph{(Bouche \cite{Bouche}, Tian \cite{tian90}, Yau \cite{yauberg}, Zelditch \cite{zelditch}, Catlin \cite{catlin}, Ruan \cite{ruan}, Lu \cite{lu}, Ma--Marinescu \cite{mm}, amongst many others)}
The Bergman function admits an asymptotic expansion
		\begin{equation*}
			\rho_k (\omega_h ) = k^n + \frac{k^{n-1}}{4 \pi } S(\omega_h ) + O(k^{n-2})
		\end{equation*}
		when $k \in \mathbb{N}$ is sufficiently large, where $S( \omega_h )$ is the scalar curvature of $\omega_h$ defined by
		\begin{equation*}
			S(\omega_h) := n \frac{\mathrm{Ric} (\omega_h ) \wedge \omega_h^{n-1}}{\omega_h^n}
		\end{equation*}
		where $\mathrm{Ric} (\omega_h )$ is the Ricci curvature of $\omega_h$, defined locally as
		\begin{equation*}
			\mathrm{Ric} (\omega_h ) := - \ai \ddbar \log \omega^n_h . 
		\end{equation*}
\end{theorem}

The theorem above has many important implications and consequences, but for the moment we simply observe the appearance of the scalar curvature in the expansion above.

We recall that a K\"ahler metric $\omega_h \in c_1 (L)$ satisfying $S (\omega_h ) = \cst$ is said to be a \textbf{constant scalar curvature K\"ahler} (or \textbf{cscK}) metric. A foundational theorem of Donaldson concerning the cscK metrics states the following.

\begin{theorem} \emph{(Donaldson \cite{donproj1})} \label{thdpr1}
Suppose that $\mathrm{Aut}_0 (X,L)$ is trivial and that $X$ admits a cscK metric $\omega_{\mathrm{cscK}}$ in the K\"ahler class $c_1 (L)$. Then, there exists $k_0 \in \nnb$ such that for any $k \in \nnb$ with $k \ge k_0$	there exists a K\"ahler metric $\omega_k \in c_1 (L)$ such that
\begin{equation*}
	\rho_k (\omega_k) = \mathrm{const.}
\end{equation*}
and $\omega_k \to \omega_{\mathrm{cscK}}$ in $C^{\infty}$ as $k \to \infty$.
\end{theorem}

In the above, $\mathrm{Aut}_0 (X,L)$ stands for the connected component of the automorphism group of $(X,L)$, i.e.~holomorphic automorphisms of $X$ which lift to the total space of $L$. A K\"ahler metric $\omega_h \in c_1 (L)$ satisfying the equation $\rho_k (\omega_h ) = \cst$ is said to be \textbf{balanced} at the level $k$.

The above theorem has an important application to the stability of $(X,L)$ in the sense of Geometric Invariant Theory, which we recall in what follows. We first recall the following important result due to H.~Luo \cite{Luo} and S.~Zhang \cite{zhang96}.

\begin{theorem} \emph{(H.~Luo \cite{Luo}, S.~Zhang \cite{zhang96})} \label{thlzh}
	Suppose that $\mathrm{Aut}_0 (X,L)$ is trivial. There exists a balanced metric at level $k$ in $c_1(L)$ if and only if $(X,L^k)$ is Chow stable.
\end{theorem}

Theorems \ref{thdpr1} and \ref{thlzh} together imply that a polarised K\"ahler manifold $(X,L)$ with discrete automorphism group which admits a cscK metric is asymptotically Chow stable, i.e.~$(X,L^k)$ is Chow stable for all large enough $k$. Chow stability that appeared in the above theorem is a classical stability notion for polarised varieties (see Remark \ref{rmchtccf}), but we present a formulation that involves test configurations that is appropriate for our discussions later.

\begin{definition}
	A \textbf{very ample test configuration $(\CX , \CL)$ of exponent $k$} for $(X , L)$ is a scheme $\CX$ endowed with a flat projective morphism $\pi : \CX \to \cx$, which is $\cx^*$-equivariant with respect to the natural $\cx^*$-action on $\cx$, with a relatively very ample Cartier divisor $\CL$ to which the action $\cx^* \actson \CX$ linearises, such that $\pi^{-1} (1) \cong X $ and $\CL |_{\pi^{-1} (1)} \cong L^k$. The preimage of $0 \in \cx$, written $\CX_0 := \pi^{-1} (0)$, is called the \textbf{central fibre}.
	
	We say that $(\CX , \CL)$ is \textbf{product} if $\CX$ is isomorphic to $X \times \cx$, and \textbf{trivial} if it is $\cx^*$-equivariantly isomorphic to $X \times \cx$ (i.e.~$\CX \isom X \times \cx$ and $\cx^*$ acts trivially on $X$).
\end{definition}

Suppose that we have a very ample test configuration $(\mathcal{X} , \mathcal{L})$ of exponent $k$. By definition $(\mathcal{X} , \mathcal{L})$ is endowed with a $\cx^*$-action. There exists an embedding $\mathcal{X} \inj \prj (H^0 (X,L^k)^{\vee})$ by \cite[Proposition 3.7]{rt07} such that the generator of the $\cx^*$-action is given by $A_k \in \mathfrak{gl} (H^0 (X,L^k))$, acting linearly on $\prj (H^0 (X,L^k)^{\vee})$, in the sense that $\CX$ is equal to the flat closure of the $\cx^*$-orbit of $X \inj \prj (H^0 (X,L^k)^{\vee}))$ generated by $A_k$ (with the $\cx^*$-action $e^{A_k t}$). Moreover, its central fibre $\CX_0 := \pi^{-1} (0)$ is equal to the flat limit of this $\cx^*$-orbit (cf.~\cite[\S 6.2]{szebook}).

Given $(\CX , \CL)$, we can construct a sequence of very ample test configurations $(\CX , \CL^{\otimes m})$ for $m \in \mathbb{N}$, each of exponent $mk$. As above we can write it as a flat closure of the $\cx^*$-orbit of $X \inj \prj(H^0 (X,L^{mk})^{\vee})$ generated by $A_{mk} \in \mathfrak{gl} (H^0 (X,L^{mk}))$, say. Since the central fibre $\CX_0$ is the flat limit of the $\cx^*$-action generated by $A_{mk}$, there is a natural $\cx^*$-action $\cx^* \actson H^0 (\CX_0, \CL^{ \otimes k}|_{\CX_0})$ generated by $A_{mk} \in \mathfrak{gl}(H^0(X,L^{mk})) = \mathfrak{gl} (H^0 (\CX_0 , \CL^{\otimes k} |_{\CX_0}))$, by noting the isomorphism $H^0 (X, L^{mk}) \isom H^0 (\CX_0 , \CL^{\otimes k} |_{\CX_0})$. 

By Riemann--Roch and equivariant Riemann--Roch, we write
\begin{align}
	\mathrm{dim} H^0 (X , L^{mk}) &= a_0 (mk)^{n} + a_1 (mk)^{n-1} + \cdots, \label{rrdimhzxlkexp} \\
	\mathrm{tr} (A_{mk}) &= b_0 (mk)^{n+1} + b_1 (mk)^n + \cdots. \label{eqrrtrarkexp}
\end{align}
Observe that $a_0$ is equal to the volume $\int_X c_1 (L)^n / n!$ of $(X,L)$.

\begin{definition} \label{defchwbtestcfg}
Let $(\mathcal{X} , \mathcal{L})$ be a very ample test configuration for $(X,L)$ of exponent $k$. The \textbf{Chow weight} of $(\mathcal{X} , \mathcal{L})$ is defined by
\begin{equation*}
	\mathrm{Chow}_k (\mathcal{X} , \mathcal{L}) := k b_0 -  \frac{a_0 \tr (A_k)}{\mathrm{dim} H^0 (X,L^k)}. 
\end{equation*}
\end{definition}

\begin{definition} \label{defchstpsssus}
A polarised K\"ahler manifold $(X,L)$ is said to be:
\begin{enumerate}
\item \textbf{Chow semistable at the level $k$} if $\mathrm{Chow}_k (\mathcal{X} , \mathcal{L}) \ge 0$ for any very ample test configuration $(\mathcal{X} , \mathcal{L})$ of exponent $k$,
\item \textbf{Chow polystable at the level $k$} if $(X,L)$ is Chow semistable at the level $k$ and $\mathrm{Chow}_k (\mathcal{X} , \mathcal{L}) = 0$ holds if and only if $(\mathcal{X} , \mathcal{L})$ is product,
\item \textbf{Chow stable at the level $k$} if $(X,L)$ is Chow semistable at the level $k$ and $\mathrm{Chow}_k (\mathcal{X} , \mathcal{L}) = 0$ holds if and only if $(\mathcal{X} , \mathcal{L})$ is trivial,
\item \textbf{Chow unstable at the level $k$} if it is not Chow semistable at the level $k$,
\end{enumerate}
\end{definition}

\begin{remark} \label{rmchtccf}
It is well-known that the Chow stability as defined above is equivalent to the more conventional definition using the Chow form of the embedded manifold $X \inj \prj(H^0 (X,L^{k})^{\vee})$; see e.g.~\cite[\S 2]{fut11} or \cite[Proposition 2.11]{mumford77}.
\end{remark}

A stability notion that is relevant to the cscK metrics is the $K$-stability introduced by Tian \cite{Tian97} and Donaldson \cite{dontoric}.

\begin{definition} \label{defofdfinvalg}
	The \textbf{Donaldson--Futaki invariant} $DF(\CX , \CL)$ is defined as 
	\begin{equation*}
	DF(\CX , \CL) = \frac{a_1 b_0 - a_0 b_1}{a_0} = \lim_{k \to \infty} \mathrm{Chow}_{mk} (\CX , \CL^{\otimes m}).
	\end{equation*}
\end{definition}

$K$-stability can be defined analogously to Chow stability, using the Donaldson--Futaki invariant instead; a crucial difference is that we need to consider test configurations of \textit{all} exponents, as opposed to a fixed one, to check $K$-stability.

\begin{definition}
	A polarised K\"ahler manifold $(X,L)$ is said to be \textbf{$K$-semistable} if $DF(\CX , \CL) \ge 0$ holds for any test configuration. $(X,L)$ is \textbf{$K$-polystable} if it is $K$-semistable and $DF(\CX , \CL) = 0$ if and only if $(\CX , \CL)$ is product, \textbf{$K$-stable} if it is $K$-semistable and $DF(\CX , \CL) = 0$ if and only if $(\CX , \CL)$ is trivial.
\end{definition}

Theorems \ref{thdpr1} and \ref{thlzh} are known to imply that a polarised K\"ahler manifold $(X,L)$ is $K$-semistable, as proved by Ross--Thomas \cite{rt07}; Donaldson \cite{donlb} later gave an alternative proof using the Calabi energy. This result can be further strengthened to yield $K$-stability, if $\mathrm{Aut}_0 (X,L)$ is trivial, by the theorem of Stoppa \cite{stokst} which also relies on the result of Arezzo--Pacard \cite{AP1}. These results can be summarised as follows.

\begin{displaymath}
			\xymatrixcolsep{6pc}\xymatrixrowsep{3pc}\xymatrix{(X,L) \text{ has a cscK metric.} \ar@{=>}[r]^-{\text{Donaldson \cite{donproj1} }} \ar@{=>}[d]_-{\text{Donaldson \cite{donlb}}} & \rho_k (\omega_k) =  \cst \  k \gg 1 . \ar@{<=>}[d]^-{\text{Luo \cite{Luo}, Zhang \cite{zhang96} }}\\
		\text{$K$-semistability}  \ar@{=>}[d]_-{\text{Stoppa \cite{stokst} }}  \ar@{<=}[r]^-{\text{Ross--Thomas \cite{rt07} }} & \text{Asymptotic Chow stability}  \\
		\text{$K$-stability} &
		}
\end{displaymath}

An improvement of the result stated above, which proves the uniform $K$-stability by employing a variational approach, was obtained by Berman--Darvas--Lu \cite{BDL}. There are many related results, e.g.~\cite{BHJ2,DR17,SD1,SD2} amongst many others, but we omit the details.

We finally recall that balanced metrics are Fubini--Study metrics, i.e.~there exists a hermitian inner product on $\cx^{N_k} = H^0(X,L^k)$ such that the Kodaira embedding
\begin{equation*}
	\iota_k : X \inj \mathbb{P}(H^0(X,L^k)^{\vee})
\end{equation*}
is an isometry with respect to the balanced metric on $X$. In other words, there exists a positive definite hermitian form $H$ on $H^0(X,L^k)$ such that the balanced metric is equal to the pullback of the Fubini--Study metric on $\mathbb{P}(H^0(X,L^k)^{\vee})$ defined by $H$.

The balanced metric is a critical point of an energy functional called the balancing energy, say $Z_k : Y_k \to \rl$, where $Y_k := GL(N_k , \cx)) / U(N_k)$ is the set of all positive definite hermitian forms on $H^0 (X,L^k)$. An important feature of the balancing energy is that it is convex along geodesics in the symmetric space $Y_k$ with respect to the standard bi-invariant metric. Moreover it is well-known that the balancing energy is strictly convex along geodesics that is not contained in the $\mathrm{Aut}_0 (X,L)$-orbit, which in particular shows that $Z_k$ is strictly convex along geodesics if $\mathrm{Aut}_0 (X,L)$ is trivial. Since $Y_k$ with the bi-invariant metric is a complete Riemannian manifold, the results that we recall and prove in Section \ref{scgcfcrm} show that $(X,L)$ admits a balanced metric at the level $k$ if and only if its asymptotic slope at infinity is strictly positive; see Figure \ref{figpic}.

\bigskip

\begin{figure}
\centering
\begin{subfigure}{0.4\textwidth}
    \includegraphics[width=\textwidth]{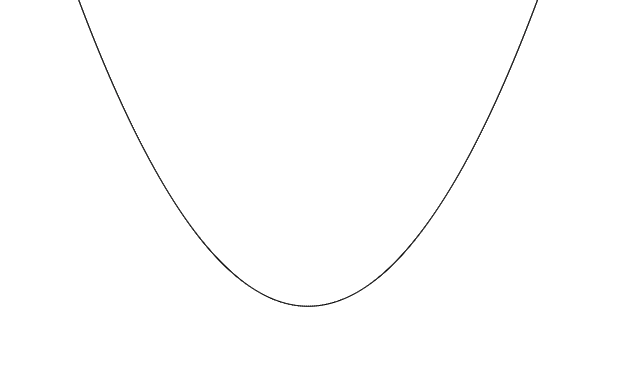}
    \caption{Convex function with a critical point}
    \label{fgstable}
\end{subfigure}
\hfill
\begin{subfigure}{0.4\textwidth}
    \includegraphics[width=\textwidth]{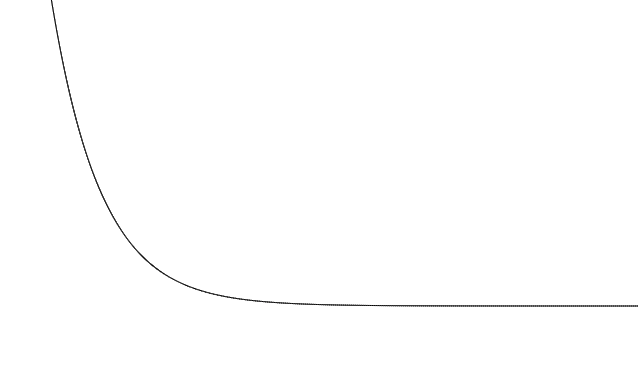}
    \caption{Bounded below but without a critical point}
    \label{fgss}
\end{subfigure}
\hfill
\begin{subfigure}{0.4\textwidth}
    \includegraphics[width=\textwidth]{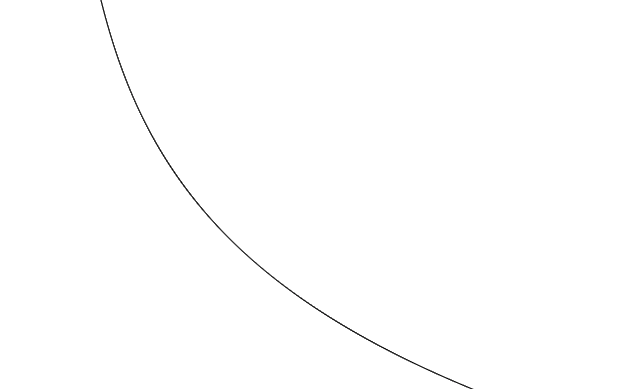}
    \caption{Unbounded below}
    \label{fguns}
\end{subfigure}
\caption{Asymptotic behaviours of convex functions}
\label{figpic}
\end{figure}

It turns out that the asymptotic slope of $Z_k$ at infinity agrees with the Chow weight (Definition \ref{defchwbtestcfg}), for geodesic rays in $Y_k$ generated by positive hermitian forms on $H^0 (X,L^k)$ with integral eigenvalues (see \cite[Proposition 3]{donlb}). We can interpret Theorem \ref{thlzh} as a claim that the asymptotic slope of $Z_k$ is positive along any geodesic rays in $Y_k$ if and only if it is positive along any geodesic rays generated by ``integral'' generators (which correspond to algebraic weights).

\section{Balanced metrics for extremal metrics}

When $\mathrm{Aut}_0 (X,L)$ is nontrivial, Theorem \ref{thdpr1} still holds as long as a series of integral invariants called the higher Futaki invariant vanishes, as proved by Futaki \cite{fut04ac} and Mabuchi \cite{mab04obs,mab05}. However, there are examples with non-vanishing higher Futaki invariants by Ono--Sano--Yotsutani \cite{osy} and Della Vedova--Zuddas \cite{dvz}, and for these manifolds we need to have a new definition of balanced metrics. It turns out that these results hold more widely for extremal metrics, which generalise cscK metrics, defined as follows.

\begin{definition}
		A K\"ahler metric $\omega \in c_1 (L)$ which satisfies 
		\begin{equation*}
		\bar{\partial}  \mathrm{grad}^{1,0}_{\omega}  S (\omega) =0 
		\end{equation*}
		is called \textbf{extremal} (i.e.~(1,0)-part of the gradient vector field $\textup{grad}_{\omega}  S (\omega) $ is a holomorphic vector field).
	\end{definition}
	
The extremal metric is a critical point of the Calabi energy, and can be regarded as a generalisation of the cscK metrics. When we have a non-cscK extremal metric, note that the automorphism group of $X$ must be non-discrete.

In what follows, we assume that $L$ is very ample and that there exists a faithful representation $\mathrm{Aut} (X,L) \inj GL(H^0(X,L))$ such that the associated linear action $\mathrm{Aut} (X,L) \actson \prj (H^0(X,L)^{\vee})$ induces the automorphism of the embedded submanifold $\iota (X) \subset \prj (H^0(X,L)^{\vee})$. It is well-known that this can be achieved by replacing $L$ by a higher tensor power.

It turns out there are multiple ways of generalising Theorem \ref{thdpr1} to the case when $\mathrm{Aut} (X,L)$ is nontrivial, as briefly recalled later. One version of them which is proved by the author is the following.

\begin{theorem} \emph{(H.~\cite{yhextremal})} \label{thextyh}
		Suppose that $X$ admits an extremal metric $\omega_{\mathrm{ext}} \in c_1 (L)$. Then, there exists $k_0 \in \nnb$ such that for any $k \in \nnb$ with $k \ge k_0$ there exists a K\"ahler metric $\omega_k \in c_1 (L)$ whose Bergman function $\rho_k (\omega_k )$ satisfies
		\begin{equation*}
			\bar{\partial} \mathrm{grad}_{\omega_k}^{1,0} \rho_k (\omega_k ) =0
		\end{equation*}
		and $\omega_k \to \omega_{\mathrm{ext}}$ in $C^{\infty}$ as $k \to \infty$.
\end{theorem}

As we saw in the previous section, the result above has implications to stability; here we need to consider the ``relative'' versions of the stability notions defined above, such as the relative $K$-stability defined by Sz\'ekelyhidi \cite{sze2007}, which essentially subtracts the contribution of the automorphism group. The statement is more complicated, and other generalisations of Theorem \ref{thdpr1} give slightly different conclusions. We summarise the known results in the following diagram; see the references cited below for the precise definitions and statements.

\begin{displaymath}
		\xymatrixcolsep{6pc}\xymatrixrowsep{4pc}\xymatrix{ & \txt{Asymptotic relative \\ Chow stability} \ar@{=>}[d]^-{\text{Mabuchi \cite{mab18} }} \\
		 (X,L) \text{ has an extremal metric.} \ar@{=>}@<4pt>[ur]^-{\text{Mabuchi \cite{mab18} \ \ \ \ \ \ }}_-{\text{\ \ \ \ \ \ \ \ Seyyedali \cite{say17} }} \ar@{=>}[r]_-{\text{Mabuchi \cite{mab04ext,mab05,mab09,mab11}}} \ar@{=>}[dr]_-{\text{H. \cite{yhextremal} }} \ar@{=>}[d]_{\text{Stoppa--Sz\'ekelyhidi \cite{stosze} }} & \txt{Asymptotic weak relative \\ Chow stability}  \ar@{<=}[d]^-{\text{H. \cite{yhstab}}  } \\
		   \text{Relative $K$-semistability} \ar@{=>}[d]_{\text{Stoppa--Sz\'ekelyhidi \cite{stosze}} } & \bar{\partial} \mathrm{grad}_{\omega_k}^{1,0} \rho_k (\omega_k ) =0 \  k \gg 1. \ar@{=>}[l]^{\text{H. \cite{yhstab} }} \\
	 	  \text{Relative $K$-stability} &
		}
\end{displaymath}
	
There is also a closely related result by Sano--Tipler \cite{santip} which is different from the above results. The results of Mabuchi \cite{mab04ext,mab05,mab09,mab11,mab18}, Seyyedali \cite{say17}, Sano--Tipler \cite{santip} can all be stated in terms of the ``weighted'' Bergman function, but not quite the Bergman function itself (they are more natural in terms of the relative moment map). The reader is referred to the papers cited above for precise statements, and to \cite[Section 6]{yhstab} and \cite[Section 6]{yhextremal} for the comparison of these results, which is rather technical; the only remark that we make here is that it is proved in \cite{yhstab} that Theorem \ref{thextyh} implies the asymptotic weak relative Chow stability and relative $K$-semistability, as indicated in the diagram above. Note also that there is another proof due to Dervan \cite{der18} for the extremal metrics implying relative $K$-stability, which applies more generally to compact (non-projective) K\"ahler manifolds.

\section{Anticanonically balanced metrics} \label{scacbal}

There is another version of balanced metrics that is adapted to Fano manifolds, which was introduced by Donaldson \cite[\S 2.2.2]{donnum09} and studied further by Berman--Boucksom--Guedj--Zeriahi \cite{BBGZ}.

	Let $(X, -K_X)$ be an (anticanonically polarised) Fano manifold. Recall that a hermitian metric $h$ on $-K_X$ naturally defines a volume form $d \mu_h$ on $X$ via the natural isomorphism $\mathcal{H}om_{C^{\infty}_X} ((-K_X) \otimes \overline{(-K_X)} , \cx ) \isom K_X \otimes \overline{K_X}$.
	\begin{definition}
		Let $\{ \sigma_i \}_{i=1}^{N_k}$ be a basis for $H^0 (X, (-K_X)^k)$ that is orthonormal with respect to the inner product $\int_X h^k ( \cdot , \cdot  ) d \mu_h $ (note the volume form).
		
		The \textbf{anticanonical Bergman function} $\rho^{\mathrm{ac}}_k (\omega_h ) \in C^{\infty} (X, \rl)$ is defined to be
		\begin{equation*}
			\rho^{\mathrm{ac}}_k (\omega_h ) := \sum_{i=1}^{N_k} | \sigma_i|^2_{h^k} .
		\end{equation*}
	\end{definition}

	\begin{definition}
		A K\"ahler metric $\omega_h \in c_1 (-K_X)$ satisfying $\rho^{\mathrm{ac}}_k (\omega_h ) = \cst$ is said to be \textbf{anticanonically balanced} at level $k$.
	\end{definition}

The ``anticanonical'' version of Donaldson's quantisation was established by Berman--Witt Nystr\"om \cite{BWN14}.

\begin{theorem} \emph{(Berman--Witt Nystr\"om \cite{BWN14})}
		Suppose that a Fano manifold $(X, -K_X)$ admits a K\"ahler--Einstein metric $\omega_{\mathrm{KE}} \in c_1 (-K_X)$ and its automorphism groups is discrete. Then there exists a sequence $\{ \omega_k \}_k$ of K\"ahler metrics in $c_1(-K_X)$ weakly converging to $\omega_{\mathrm{KE}}$ in the sense of currents, whose anticanonical Bergman function is constant for all large enough $k$, i.e.~$\rho^{\mathrm{ac}}_k (\omega_k) = \cst$ for $k \gg 1$.
\end{theorem}
	
While we focus on the case when 	$\mathrm{Aut}_0 (X) = \mathrm{Aut}_0 (X, -K_X)$ is trivial, Berman--Witt Nystr\"om \cite{BWN14} proved an analogous result for the K\"ahler--Ricci $g$-solitons when $\mathrm{Aut}_0 (X)$ is nontrivial. They proved the convergence in terms of currents, but this can be improved to the smooth convergence by Takahashi \cite[Theorem 1.3 with $N=1$]{tak19} and Ioos \cite{Ioos20,Ioos21}. The result recently proved by the author is the following, which can be regarded as an anticanonical version of Theorem \ref{thlzh}.

\begin{theorem} \emph{(H.~\cite{acbal})} \label{mthst}
Let $k \in \mathbb{N}$ be large enough such that $(-K_X)^k$ is very ample. A Fano manifold $(X , -K_X)$ with the discrete automorphism group admits a unique anticanonically balanced metric at level $k$ if and only if it satisfies the following stability condition: for any very ample test configuration $( \CX , \CL )$ for $(X , -K_X)$ of exponent $k$ we have $\mathrm{Ding} ( \CX , \CL ) + \mathrm{Chow}_k ( \CX , \CL ) \ge 0$, with equality if and only if $( \CX , \CL )$ is trivial.
\end{theorem}

One direction of the above result, i.e.~the existence of anticanonically balanced metrics implying the stated stability condition, was proved by Saito--Takahashi \cite[Theorem 1.2]{st19}. Thus the main point of the above theorem is the stability implying the anticanonically balanced metrics, although the proof in \cite{acbal} easily establishes both directions.

The definition of the Ding invariant that appeared in the above statement is as follows.

\begin{definition} \label{dfding}
	Let $(\CX , \CL)$ be a very ample test configuration for $(X , -K_X)$ of exponent $k$, and $\nu : \CX^{\nu} \to \CX$ be its normalisation with $\CL^{\nu} := \nu^* \CL$. Let $D_{(\CX^{\nu} , \CL^{\nu})}$ be a $\rtn$-divisor on $\CX^{\nu}$, whose support is contained in $\CX^{\nu}_0 := (\nu \circ \pi )^{-1} (0)$, such that $-k (K_{\CX^{\nu} / \cx } + D_{(\CX^{\nu} , \CL^{\nu})})$ is a Cartier divisor corresponding to $\CL^{\nu}$; it is well-known that such a $\rtn$-divisor $D_{(\CX^{\nu} , \CL^{\nu})}$ exists uniquely. The \textbf{Ding invariant} of $(\CX , \CL)$ is a real number defined by
	 \begin{equation*}
	 	\mathrm{Ding} (\CX , \CL) := - \frac{b_0}{a_0} - 1 + \mathrm{lct} (\CX^{\nu} , D_{(\CX^{\nu} , \CL^{\nu})} ; \CX^{\nu}_0),
	 \end{equation*}
	 where $\mathrm{lct} (\CX^{\nu} , D_{(\CX^{\nu} , \CL^{\nu})} ; \CX^{\nu}_0)$ is the log canonical threshold of $\CX^{\nu}_0$ with respect to $(\CX^{\nu} , D_{(\CX^{\nu} , \CL^{\nu})})$, which is defined as
	 \begin{equation*}
		\mathrm{lct} (\CX^{\nu} , D_{(\CX^{\nu} , \CL^{\nu})} ; \CX^{\nu}_0) := \sup \{ c \in \rl \mid (\CX^{\nu} , D_{(\CX^{\nu} , \CL^{\nu})} +c \CX^{\nu}_0) \text{ is sub log canonical}. \} .
	\end{equation*}
\end{definition}

We can define the Ding stability by using the above invariant, which is an important stability notion concerning the K\"ahler--Einstein metrics, but we omit the details.

The main application of Theorem \ref{mthst}, combined with a result by Rubinstein--Tian--Zhang \cite{rtz}, is the following.

\begin{corollary} \emph{(Rubinstein--Tian--Zhang \cite{rtz}, H.~\cite{acbal})}\label{mthstc}
	Let $(X,-K_X)$ be a Fano manifold with the discrete automorphism group. For $k \in \mathbb{N}$ large enough such that $(-K_X)^k$ is very ample, the $\delta_k$-invariant of Fujita--Odaka satisfies $\delta_k >1$ if and only if $\mathrm{Ding} ( \CX , \CL ) + \mathrm{Chow}_k ( \CX , \CL ) >0$ for any nontrivial very ample test configuration $( \CX , \CL )$ for $(X , -K_X)$ of exponent $k$.
\end{corollary}

The $\delta_k$-invariant is an algebraic invariant defined by Fujita--Odaka \cite{fo18} and the above statement is purely algebro-geometric, but it seems that no purely algebro-geometric proof is known at the moment of writing this paper; the proof in \cite{acbal} relies on Theorem \ref{mthst} and \cite[Theorem 2.3]{rtz} which concern the anticanonically balanced metrics in differential geometry. The summary of the results so far are as follows.

\begin{displaymath}
		\xymatrixcolsep{10pc}\xymatrixrowsep{8pc}\xymatrix{ 
		\mathrm{Ding} (\CX , \CL) + \mathrm{Chow}_k (\CX , \CL) >0 \ar@{<=}@< 4pt>[d]^-{\text{Saito--Takahashi \cite{st19} }} \ar@{=>}@<-4pt>[d]_-{\text{H.~\cite{acbal} }} \ar@{<=>}[r]_-{\text{No algebraic proof known}} & \delta_k > 1 \\
		\exists \text{ anticanonically balanced metric}  \ar@{<=>}[ur]_(.4){\quad \quad \quad \text{Rubinstein--Tian--Zhang \cite{rtz} }} & }
\end{displaymath}

Note that an analogous result for K\"ahler--Einstein metrics and stability is known, in which we replace $\delta_k$ by the delta invariant $\delta:= \lim_{k \to \infty} \delta_k$ and $\mathrm{Ding} (\CX , \CL) + \mathrm{Chow}_k (\CX , \CL) >0$ by the uniform Ding stability; see Berman--Boucksom--Jonsson \cite{BBJ}. In this case, the equivalence between $\delta >1$ and the uniform Ding stability can be proved by methods in algebraic geometry \cite{bj20,fo18}.

Similarly to the usual balanced metrics, the anticanonically balanced metrics are Fubini--Study metrics. There exists an ``anticanonical'' version of the balancing energy, say $Z^{\mathrm{ac}}_k : Y_k \to \rl$, where $Y_k = GL(N_k , \cx)) / U(N_k)$ is the set of all positive definite hermitian forms on $H^0 (X,(- K_X)^k)$. $Z^{\mathrm{ac}}_k$ is convex along geodesics on $Y_k$ and strictly convex along geodesics that are not contained in the $\mathrm{Aut}_0 (X)$-orbit. Thus, just as discussed at the end of Section \ref{scitr}, the anticanonically balanced metric exists if and only if the asymptotic slope of $Z^{\mathrm{ac}}_k$ is strictly positive along geodesics (by the results in Section \ref{scgcfcrm}). Similarly to the case of the usual balanced metrics, the asymptotic slope of $Z^{\mathrm{ac}}_k$ at infinity agrees with the invariant $\mathrm{Ding} (\CX , \CL) + \mathrm{Chow}_k (\CX , \CL)$, which was first observed by Saito--Takahashi \cite{st19}, for geodesic rays in $Y_k$ generated by positive hermitian forms on $H^0 (X,(-K_X)^k)$ with integral eigenvalues. Theorem \ref{mthst} is proved by showing that the asymptotic slope of $Z^{\mathrm{ac}}_k$ is positive along any geodesic rays in $Y_k$ if and only if it is positive along any geodesic rays generated by integral generators.

\section{Geodesically convex functions on complete Riemannian manifolds} \label{scgcfcrm}

The balancing energy and its anticanonical version play a very important role in the study of (anticanonically) balanced metrics, whose important feature is the convexity along geodesics. We remarked at the end of Sections \ref{scitr} and \ref{scacbal} that the critical point of the (anticanonical) balancing energy exists if and only if its asymptotic slope is strictly positive along geodesics (see Figure \ref{figpic}). While this fact is surely well-known to the experts, it seems that its detailed proof is not readily available in the literature. In what follows, we provide a detailed proof which relies on the Hopf--Rinow theorem. In fact, Hopf--Rinow is the only significant ingredient in the proof, and it can be generalised to any length metric space for which a generalisation of the Hopf--Rinow theorem is known to hold. Thus the proof works not only for (finite dimensional) complete Riemannian manifolds such as $Y_k = GL(N_k , \cx)) / U(N_k)$, but also for locally compact complete length spaces (see Remark \ref{rmcrmclsp}).

\begin{definition}
	A Riemannian manifold $(Y,\textsl{g}_Y)$ is said to be \textbf{complete} if every geodesic of $(Y,\textsl{g}_Y)$ can be extended indefinitely.
\end{definition}

In what follows, $d_Y$ stands for the distance function defined by the Riemannian metric $\textsl{g}_Y$. A foundational theorem in Riemannian geometry is the Hopf--Rinow theorem stated as follows (see e.g.~\cite[Chapter IV, Section 4]{KN1}).

\begin{theorem} \emph{(Hopf--Rinow)} \label{thmhopfrinow}
	Let $(Y,\textsl{g}_Y)$ be a connected Riemannian manifold. The following are equivalent.
	\begin{enumerate}
	\item $(Y,\textsl{g}_Y)$ is a complete Riemannian manifold.
	\item There exists a point $y_0 \in Y$ such that any geodesic segment $\gamma (t)$ emanating from $y_0$ can be extended indefinitely, i.e.~can be extended to a geodesic ray $\gamma : [0, + \infty) \to Y$, $\gamma (0) = y_0$;
	\item $(Y, d_Y)$ is a complete metric space.
	\item A subset of $Y$ is compact if and only if it is closed and bounded with respect to $d_Y$.
\end{enumerate}
Moreover, any of the above conditions implies that any two points in $Y$ can be joined by a minimising geodesic.
\end{theorem}

\begin{remark} \label{rmcrmclsp}
	The Hopf--Rinow theorem above, especially the fourth item (the Heine--Borel property), is the \textit{only} property of a complete Riemannian manifold that we will use in this section. In other words, we shall provide a proof of the main results (Theorems \ref{thcxffmp} and \ref{thcxffmpg}) in such a way that it generalises to any length metric space for which the analogue of Hopf--Rinow holds: for example, it is well-known that Hopf--Rinow can be generalised to any locally compact length space, which is called the Hopf--Rinow--Cohn--Vossen Theorem in \cite[Theorem 2.5.28]{BBI}.
\end{remark}

Note the following easy consequence of the theorem above.

\begin{lemma} \label{lmprestift}
	A continuous function $f : Y \to \rl$ is proper, i.e.~$f^{-1} (I)$ is compact in $(Y, d_Y)$ for any closed interval $I \subset \rl$, if and only if the sequence $\{ f (y_i ) \}_i \subset \rl$ is unbounded for any $d_Y$-unbounded sequence $\{ y_i \}_i$.
\end{lemma}

\begin{proof}
	If $f$ is proper and $\{ f(y_i ) \}_i$ is a bounded sequence, we find a closed interval $I$ that contains $\{ f(y_i ) \}_i$, which implies that the sequence $\{ y_i \}_i$ is contained in the compact set $f^{-1} (I)$ and hence must be $d_Y$-bounded.
	
	Suppose conversely that $\{ f (y_i ) \}_i \subset \rl$ being bounded implies that $\{ y_i \}_i$ is $d_Y$-bounded. For a closed interval $I \subset \rl$ pick any sequence $\{ y_i \}_i \subset f^{-1} (I)$, which must be $d_Y$-bounded by hypothesis. On the other hand, since $f^{-1} (I)$ is closed, Theorem \ref{thmhopfrinow} implies that there exists a subsequence of $\{ y_i \}_i$ that converges to a point in $f^{-1} (I)$, proving its compactness.
\end{proof}

For any $y_0 \in Y$ and $\tau > 0$, we define a subset of $Y$ defined as
\begin{equation*}
	S(y_0, \tau ) := \{ y \in Y \mid d_Y (y,y_0) = \tau \} = \{ \gamma (t) \mid \text{geodesic } \gamma \text{ with } \gamma (0) = y_0 \}
\end{equation*}
which is compact by Theorem \ref{thmhopfrinow}.

We first define a geodesically convex function, which is nothing but a function that is convex along geodesic segments.
\begin{definition}
A real-valued function $f$ on a Riemannian manifold $(Y,\textsl{g}_Y)$	is said to be \textbf{convex along geodesics} if for any geodesic segment $\{ \gamma (\tau) \}_{\tau_0 \le \tau \le \tau_1}$ we have
\begin{equation*}
		f(\gamma (\tau ) ) \le \frac{\tau - \tau_0}{\tau_1 - \tau_0}  f(\gamma (\tau_1)) + \frac{\tau_1-\tau}{\tau_1 - \tau_0} f ( \gamma (\tau_0) ).
\end{equation*}
$f$ is said to be \textbf{strictly convex along geodesics} if a strict inequality holds in the above.
\end{definition}

We first recall some elementary facts on convex functions, which can be found e.g.~in \cite{Horm}. First, if $f$ is a convex function defined on a closed interval $[s_1, s_2]$, the left derivative $f'_l (\tau)$ and the right derivative $f'_r (\tau)$ are both well-defined for $\tau \in (s_1, s_2)$. These are in fact monotonically increasing in $\tau$ since for all $\tau_1, \tau_2 \in (s_1 , s_2)$, $t_1 < t_2$, we have
\begin{equation} \label{cvxlrdv}
	f'_l (\tau_1) \le f'_r (\tau_1) \le \frac{f(\tau_2) - f(\tau_1)}{\tau_2 - \tau_1} \le f'_l (\tau_2) \le f'_r (\tau_2)
\end{equation}
by \cite[Corollary 1.1.6]{Horm}. If $f$ is continuous on $[s_1 , s_2]$ and $C^2$ in the interior, the convexity is equivalent to $f''(\tau) \ge 0$ for $\tau \in (s_1 , s_2)$ \cite[Corollary 1.1.10]{Horm}. The following lemma is also well-known; see e.g.~\cite[Theorem 1.1.5]{Horm}.
\begin{lemma} \label{lmmtdqhm}
	Suppose that $f : \rl \to \rl$ is convex. Then, for any $x \in \rl$, $(f(x+\tau) - f(x)) / \tau$ is increasing in $\tau$.
\end{lemma}

We observe the following elementary consequence of the above results.
\begin{lemma} \label{lmasift}
	Suppose that $f : \rl \to \rl$ is convex. Then
	\begin{equation*}
		\lim_{\tau \to + \infty} \frac{f(\tau)}{\tau} = \lim_{\tau \to + \infty} f'_l (\tau) = \lim_{\tau \to + \infty} f'_r (\tau),
	\end{equation*}
	where we allow the limit to be $+ \infty$. The same result holds for the limit $\tau \to - \infty$, by allowing it to be $- \infty$.
\end{lemma}

All the results above carry over to geodesically convex functions in an obvious manner. Recall also the following well-known result.

\begin{lemma} \label{lmecvcpgmin}
	Suppose that $f : Y \to \rl$ is a $C^1$ function that is convex along geodesics. Then any critical point of $f$ necessarily attains the global minimum of $f$ over $Y$.
\end{lemma}

We first prove the following technical lemma, which will be useful later.

\begin{lemma} \label{lmtccvsp}
	Let $(Y,\textsl{g}_Y)$ be a complete Riemannian manifold and $y_0 \in Y$ be a fixed point. Let $f : Y \to \rl$ be a continuous function that is convex along geodesic rays in $(Y,\textsl{g}_Y)$ emanating from $y_0$. Suppose that a sequence $\{ y_i \}_i \subset Y$ satisfies the following.
	\begin{enumerate}
		\item $\{ y_i \}_i \subset Y$ is $d_Y$-unbounded.
		\item There exists a geodesic $\gamma_*$ in $(Y, d_Y)$ such that the following hold:
		\begin{itemize}
		\item when we write $y_i = \gamma_i (\tau_i)$ by using a geodesic $\gamma_{i}$ and some $\tau_i >0$, we have $d_Y (\gamma_i (1) , \gamma_* (1)) \to 0$ as $i \to \infty$;
		\item $\displaystyle{\lim_{\tau \to + \infty} \frac{f(\gamma_{*} (\tau))}{\tau} > 0}$.
		\end{itemize}
	\end{enumerate}
	Then, we have  $\displaystyle{\liminf_{i \to \infty} \frac{f(y_i)}{\tau_i}} > 0$.
\end{lemma}

\begin{proof}
	Set $\lim_{\tau \to + \infty} f(\gamma_{*} (\tau)) /\tau =:c >0$, which we allow to be $+ \infty$. This implies that, if $c$ is finite, for any $\epsilon >0$ there exists $\tau_0 ( \epsilon ) > 1$ such that for all $\tau \ge \tau_0 (\epsilon )$ we have $f (\gamma_{*} (\tau)) \ge \tau (c - \epsilon )$. The argument below carries over word by word to the case $c = + \infty$, by simply replacing $c - \epsilon$ by an arbitrarily large real number.
	
	Suppose for contradiction that $\liminf_{i \to \infty} f(y_i)/\tau_i =: c' \le 0$. We now take $\epsilon>0$ to be sufficiently small and $\tau_1 := \tau_1 (c, c' , \epsilon ) > \tau_0 ( \epsilon )$ to be large enough so that
	\begin{equation*}
		(c - \epsilon ) - \frac{f(y_0) + 1}{\tau_1} > c' + \frac{c}{2},
	\end{equation*}
	which is well-defined since $c' \le 0$. On the other hand, by the continuity of $f$, there exists $\delta = \delta (\tau_1 (c, c', \epsilon )) >0$ that depends on $\tau_1 (c, c', \epsilon )$ such that
	\begin{equation*}
		d_Y (\gamma_i (1) , \gamma_* (1)) < \delta \Rightarrow f(\gamma_{i} (\tau_1)) > f(\gamma_{*} (\tau_1)) -1.	
	\end{equation*}
	To prove the above assertion, observe that for any $\epsilon' >0$ there exist $\epsilon , \delta >0$ such that $d_Y (\gamma_i (1) , \gamma_* (1) ) < \delta \Rightarrow d_Y (\gamma_i (1 + \epsilon ) , \gamma_* (1 + \epsilon ) ) < \epsilon'$, and iterate using the compactness of $[1,\tau_1]$.
	
	Note further that by convexity (see Lemma \ref{lmmtdqhm}) we have
	\begin{equation*}
		\frac{f(\gamma_{i} (\tau)) - f(y_0)}{\tau} \ge \frac{f(\gamma_{i} (\tau_1)) - f(y_0)}{\tau_1}
	\end{equation*}
	for all $\tau > \tau_1$. We take $i$ to be large enough so that $\tau_i > \tau_1 (c, c' , \epsilon )$ holds and $d_Y (\gamma_i (1) , \gamma_* (1) ) < \delta$, and then we find
	\begin{align*}
		\frac{f(\gamma_{i} (\tau_i)) - f(y_0)}{\tau_i} &\ge \frac{f(\gamma_{i} (\tau_1)) - f(y_0)}{\tau_1} \\
		&> \frac{f(\gamma_{*} (\tau_1)) - f(y_0) -1}{\tau_1} \\
		&\ge \frac{\tau_1( c - \epsilon ) - f(y_0) -1}{\tau_1} \\
		&> c'+c/2.
	\end{align*}
	Thus, for all large enough $i$, we get
	\begin{equation*}
		\frac{f(y_i)}{\tau_i} > c' + \frac{c}{2} + \frac{f (y_0)}{\tau_i} > c' + \frac{c}{4},
	\end{equation*}
	which contradicts $\liminf_{i \to \infty} f(y_i)/\tau_i =: c'$.
\end{proof}

We now give the proof of the main result of this section, following the approach in \cite[Theorem 1.6]{Bou18icm}.

\begin{theorem} \label{thcxffmp}
	Let $(Y,\textsl{g}_Y)$ be a complete Riemannian manifold and $y_0 \in Y$ be a fixed point. Let $f : Y \to \rl$ be a continuous function that is strictly convex along geodesic rays in $(Y,\textsl{g}_Y)$. Then the following are equivalent.
	\begin{enumerate}
		\item $f$ admits a unique global minimum over $Y$.
		\item $f : Y \to \rl$ is proper and bounded from below.
		\item There exists $C, D >0$ such that $f (y) \ge C d_Y (y, y_0) -D$ for all $y \in Y$.
		\item $\displaystyle{\lim_{\tau \to + \infty} \frac{f(\gamma(\tau))}{\tau} >0}$ for any geodesic $\gamma$ emanating from $y_0$.
		\item There exists $C >0$ such that $\displaystyle{\lim_{\tau \to + \infty} \frac{f(\gamma(\tau))}{\tau} > C}$ for any geodesic $\gamma$ emanating from $y_0$.
	\end{enumerate}
\end{theorem}

\begin{proof}
	It suffices to prove $3 \Rightarrow 2 \Rightarrow 1 \Rightarrow 4 \Rightarrow 3$, since $5 \Rightarrow 4$ is obvious and $3 \Rightarrow 5$ follows from
	\begin{equation*}
		\frac{f (\gamma (\tau))}{\tau} \ge C \frac{d_Y (\gamma (\tau) , y_0) }{d_Y (\gamma (\tau) , \gamma (0))} - \frac{D}{\tau} \ge C \frac{d_Y (\gamma (\tau) , \gamma (0) ) - d_Y (y_0 , \gamma (0) ) }{d_Y (\gamma (\tau) , \gamma (0))} - \frac{D}{\tau} ,
	\end{equation*}
	and taking the limit of $\tau = d_Y (\gamma (\tau) , \gamma (0) ) \to + \infty$.

	$3 \Rightarrow 2$ is an obvious consequence of Lemma \ref{lmprestift}. To prove $2 \Rightarrow 1$, take a sequence $\{ y_i \}_i \subset Y$ such that $f (y_i) \to \inf_Y f =:c > - \infty$ as $i \to \infty$. Since $f^{-1} ([c, c+1])$ is compact, this contains a convergent subsequence, which must attain the global minimum of $f$. We prove the uniqueness: suppose that there exist $y_1 , y_2 \in Y$, $y_1 \neq y_2$, that satisfy $f(y_1) = f(y_2) = \inf_Y f$. We pick a geodesic segment $\{ \gamma (\tau) \}_{0 \le \tau \le 1}$ connecting $y_1 = \gamma (0)$ and $y_2 = \gamma (1)$. On the other hand, strict convexity means
	\begin{equation*}
		f(\gamma (1/2) ) < \frac{1}{2} f(\gamma(0)) + \frac{1}{2} f(\gamma (1)) = \inf_Y f,
	\end{equation*}
	which is a contradiction, and hence we get $y_1 =y_2$.

We prove $1 \Rightarrow 4$. Since $f$ is strictly increasing along geodesics in a neighbourhood of the unique global minimum $y_* \in Y$, we find $f'_l (\gamma (0)) >0$ for any geodesic that emanates from $y_*$. Since the left (and right) derivatives of $f$ are monotonically increasing by (\ref{cvxlrdv}), we get the claim for such geodesics by Lemma \ref{lmasift}, and find in particular that $f (\gamma (\tau))$ is strictly increasing in $\tau$ for any geodesic emanating from $y_*$.

We now claim that for each geodesic ray $\{ \gamma (\tau) \}_{\tau \ge 0}$, passing through $y_0$ (or indeed any point) but not necessarily through $y_*$, there exists $\tau_*$ such that $f (\gamma (\tau_*)) = \inf_{\tau \ge 0} f(\gamma (\tau)) > - \infty$; given such claim, we get the result $1 \Rightarrow 4$ by exactly the same argument as before, by using the strict geodesic convexity of $f$. 

Suppose for contradiction that we have an unbounded sequence $\{ \tau_i \}_i \subset \rl_{\ge 0}$ such that $f( \gamma (\tau_i)) \to \inf_{\tau \ge 0} f(\gamma (\tau)) =: m_{\gamma}$ as $i \to \infty$. We then find that the $d_Y$-unbounded sequence $\{ y_i \}_{i} \subset Y$, $y_i := \gamma (\tau_i)$, satisfies $f(y_i) \to m_{\gamma}$. Since $(Y,\textsl{g}_Y)$ is complete, we can connect each $y_i$ to $y_*$ by a geodesic $\{ \gamma_{i} ( s ) \}_{s \ge 0}$, say, with $\gamma_{i} (0) = y_*$ and $y_i = \gamma_{i} (s_i)$. By passing to a subsequence if necessary, we may assume that $\gamma_i (1) \in S(y_* , 1)$ converges to some $\gamma_* (1) \in S(y_*, 1)$ by recalling the compactness of $S(y_*, 1)$. Recalling that $\{ y_i \}_i$ is $d_Y$-unbounded, and also that we know $\lim_{s \to + \infty} f( \gamma_{*} (s) )/s =:c >0$ by the argument above (we assume $c$ is finite, but the case $c = + \infty$ can be treated similarly), we apply Lemma \ref{lmtccvsp} to conclude $\liminf_{i \to \infty} f(y_i)/\tau_i >0$. However, $f (y_i) \to m_{\gamma}$ ($i \to \infty$) would imply that $\limsup_{i \to \infty} f(y_i)/\tau_i  = \lim_{i \to \infty} f(y_i)/\tau_i =0$, which is a contradiction.

		We finally prove $4 \Rightarrow 3$. Suppose for contradiction that for any $C,D >0$ there exists $y \in Y$ such that $f(y) < C d_Y (y, y_0) -D$. Pick a sequence $\{ C_i \}_i, \{ D_i \}_i \subset \rl_{>0}$ such that $C_i \to 0$ and $D_i \to + \infty$ as $i \to \infty$, and also $\{ y_i \}_i \subset Y$ so that
		\begin{equation*}
			f(y_i) < C_i d_Y (y_i, y_0) - D_i.
		\end{equation*}
		This means that $\{ y_i \}_i$ cannot be contained in a $d_Y$-bounded subset in $Y$, since otherwise we would have $f(y_i) < C' - D_i \to - \infty$ as $i \to \infty$, for some constant $C' >0$ that does not depend on $i$, which is a contradiction since a $d_Y$-bounded subset is compact by Theorem \ref{thmhopfrinow} and hence $\inf_i f(y_i)$ should be finite.
		
		We now choose a geodesic $\{ \gamma_{i} (\tau) \}_{\tau \ge 0}$ with $y_i = \gamma_{i} (\tau_i)$, and $\gamma_{i} (0) = y_0$. Note $\tau_i = d_Y (y_i, y_0)$. Since $\{ y_i \}_i$ is $d_Y$-unbounded we find that $\{ \tau_i \}_i \subset \rl_{\ge 0}$ is unbounded, and by taking a subsequence we further assume that it is monotonically increasing. Moreover, since $S (y_0 , 1)$ is compact, we may further assume that $\gamma_i (1)$ converges to $\gamma_* (1) \in S (y_0 , 1)$. We then have $\liminf_{i \to \infty} f(y_i) / \tau_i >0$ by Lemma \ref{lmtccvsp}, but on the other hand
		\begin{equation*}
			\frac{f (y_i)}{\tau_i} < C_i - \frac{D_i}{\tau_i}
		\end{equation*}
		implies
		\begin{equation*}
			\limsup_{i \to \infty} \frac{f (y_i)}{\tau_i} \le \limsup_{i \to \infty} C_i - \liminf_{i \to \infty} \frac{D_i}{\tau_i} \le 0,
		\end{equation*}
		which is a contradiction.
\end{proof}

\begin{remark} \label{rmstgdy0}
	In the fourth and fifth condition of Theorem \ref{thcxffmp}, the geodesic $\gamma$ need not emanate from $y_0$. To see this, we only need to check that the first condition implies $\lim_{\tau \to + \infty} f(\gamma(\tau)) / \tau >0$ for all geodesics not necessarily emanating from $y_0$, which is stronger than the fourth, but this is obvious from the proof of $1 \Rightarrow 4$.
\end{remark}

We also have a group equivariant version, which can deal with the balanced metrics when it is applied to the case where $Y= Y_k = GL(N_k , \cx)) / U(N_k)$ and $\mathrm{Aut}_0(X,L)$ is nontrivial; see \cite[Sections 2.3 and 5.1]{acbal} for more details. Let $H$ be any group, and suppose that $Y$ admits an $H$-action which is a $d_Y$-isometry; for the sake of the exposition we assume that this is the right action but the argument is exactly the same when it is switched to the left action. We define
\begin{equation*}
	d_{Y, H} (y_1,y_2) := \inf_{h_1 , h_2 \in H} d_Y (y_1 \cdot h_1 , y_2 \cdot h_2),
\end{equation*}
which is sometimes called the \textbf{reduced distance}.

We also say that a geodesic $\gamma$ is \textbf{parametrised by the reduced arc-length} if $\tau = d_{Y,H} (\gamma (\tau) , \gamma (0))$.

\begin{theorem} \label{thcxffmpg}
	Let $(Y,\textsl{g}_Y)$ be a complete Riemannian manifold with the isometric $H$-action, and $y_0 \in Y$ be a fixed point in $Y$. Let $f : Y \to \rl$ be a $H$-invariant continuous function that is strictly convex along geodesic rays in $(Y,\textsl{g}_Y)$ that is not contained in the $H$-orbit, i.e. 
	\begin{equation*}
		f(\gamma (\tau) ) \le \frac{\tau - \tau_0}{\tau_1 - \tau_0}  f(\gamma (\tau_1)) + \frac{\tau_1-\tau}{\tau_1 - \tau_0} f ( \gamma (\tau_0) ),
	\end{equation*}
	for $\tau \in [\tau_0 , \tau_1]$, with equality if and only if $\{ \gamma (\tau ) \}_{\tau_0 \le \tau \le \tau_1}$ is contained in $\gamma (\tau_0) \cdot H$. The following are equivalent.
	
	\begin{enumerate}
		\item There exists $y_* \in Y$ which attains the global minimum of $f$, which is unique modulo the $H$-action.
		\item There exists $C, D >0$ such that $f (y) \ge C d_{Y, H} (y, y_0) -D$.
		\item $\displaystyle{\lim_{\tau \to + \infty} \frac{f(\gamma(\tau))}{\tau} >0}$ for any geodesic $\gamma$ emanating from $y_0$, not contained in the $H$-orbit.
		\item There exists $C >0$ such that $\displaystyle{\lim_{\tau \to + \infty} \frac{f(\gamma(\tau))}{\tau} > C }$ for any geodesic $\gamma$ emanating from $y_0$ and not contained in the $H$-orbit, parametrised by the reduced arc-length.
	\end{enumerate}
\end{theorem}

As pointed out in Remark \ref{rmstgdy0}, the geodesics in the third and fourth items need not emanate from $y_0$.

\begin{proof}
	We prove $2 \Rightarrow 1 \Rightarrow 3 \Rightarrow 2$, as $4 \Rightarrow 3$ and $2 \Rightarrow 4$ are obvious by arguing as before in the proof of Theorem \ref{thcxffmp}.

	The proof of $2 \Rightarrow 1$ is just an adaptation of Lemma \ref{lmprestift}: noting that $f$ is bounded below, we pick a minimising sequence $\{ y_i \}_i$ so that $f(y_i) \to \inf_Y f$ as $i \to \infty$. We then find that $\{ d_{Y, H} (y_i, y_0) \}_i \subset \rl_{\ge 0}$ is a bounded sequence. Thus there exists $\{ h'_i \}_i , \{ h_i \}_i \subset H$ such that $\{ d_Y (y_i \cdot h'_i , y_0 \cdot h_i) \}_{i}$ is also a bounded sequence. Now, since $H$ acts on $Y$ by isometry, we find $d_Y (y_i \cdot h'_i , y_0 \cdot h_i) = d_Y (y_i \cdot (h'_i h_i^{-1}) , y_0)$. Hence defining $y'_i := y_i \cdot (h'_i h_i^{-1})$, we get a $d_Y$-bounded sequence $\{ y'_i \}_i \subset Y$, which must contain a subsequence that converges in $Y$ by Theorem \ref{thmhopfrinow}. The limit of any such subsequence must attain $\inf_Y f$, since $f (y_i) = f(y'_i)$ by the $H$-invariance, proving the existence of the global minimiser. To show the uniqueness modulo the $H$-action, we argue as before: pick any $y_1, y_2 \in Y$ with $y_1 \neq y_2$ and $f(y_1 ) = f(y_2) = \inf_Y f$, and a geodesic $\gamma$ connecting them with $y_1 = \gamma (0)$ and $y_2 = \gamma (1)$. We then find
	\begin{equation*}
		f(\gamma (\tau)) \le (1-\tau) f(\gamma(0)) + \tau f(\gamma (1)) = \inf_Y f
	\end{equation*}
	for all $0 \le \tau \le 1$. The above inequality must clearly be an equality, which forces $\gamma ( \tau ) \subset y_1 \cdot H$ as required.
	
	The other implications are almost exactly the same as the proof of Theorem \ref{thcxffmp}, but we supplement some details.
	
	We prove $1 \Rightarrow 3$. Since $f$ is strictly increasing along geodesics not contained in the $H$-orbit, we find $f'_l (\gamma (0)) >0$ for any geodesic that emanates from the global minimum $y_*$ which is not contained in the $H$-orbit. Since the left (and right) derivatives of $f$ are monotonically increasing by (\ref{cvxlrdv}), we get the claim for such geodesics by Lemma \ref{lmasift}. We then prove, just as we did in Theorem \ref{thcxffmp}, that for each geodesic ray $\{ \gamma (\tau) \}_{\tau \ge 0}$, not necessarily passing through $y_*$ and not contained in the $H$-orbit, there exists $\tau_*$ such that $f (\gamma (\tau_*)) = \inf_{\tau \ge 0} f(\gamma (\tau)) > - \infty$. The only modification required in the proof is that we need to make sure that each $y_i$ can be connected to $y_*$ by a geodesic that is not contained in the $H$-orbit; this is straightforward, since otherwise there would exist $h \in H$ such that $f(y_i) = f(y_* \cdot h) = f(y_*) = \inf_Y f$, by the $H$-invariance of $f$, which must attain $\inf_{\tau \ge 0} f(\gamma (\tau))$. This gives us the desired result.
	
	The proof of $3 \Rightarrow 2$ is also similar. We pick a sequence $\{ C_i \}_i, \{ D_i \}_i \subset \rl_{>0}$ such that $C_i \to 0$ and $D_i \to + \infty$ as $i \to \infty$, and also $\{ y_i \}_i \subset Y$ so that
		\begin{equation*}
			f(y_i) < C_i d_{Y,H} (y_i, y_0) - D_i,
		\end{equation*}
		which means that $\{ y_i \}_i$ cannot be contained in a $d_{Y,H}$-bounded subset in $Y$, since as we saw in the proof of $2 \Rightarrow 1$ we can define $y'_i := y_i \cdot h_i$ for some $\{ h_i \}_i \subset H$ such that for all large enough $i$ we have
		\begin{equation*}
			f(y'_i) < C_i d_{Y} (y'_i, y_0) - D_i+1.
		\end{equation*}
		This means that the sequence $\{ y'_i \}_i \subset Y$ must be $d_Y$-unbounded, as otherwise it would contradict Theorem \ref{thmhopfrinow}. We note that we can connect each $y'_i$ to $y_0$ by a geodesic that is not contained in the $H$-orbit (as otherwise we would have $f(y'_i) = f(y_0 \cdot h_i) = f(y_0)$), the argument for Theorem \ref{thcxffmp} carries over word by word to derive the required contradiction.
\end{proof}

We finally summarise what we have proved so far in the form that is helpful for the study of the (anticanonical) balancing energy defined on the symmetric space $Y=Y_k = GL(N_k , \cx)) / U(N_k)$, which is a complete Riemannian manifold with respect to the bi-invariant metric.

\begin{corollary} \label{crcxffmpg}
	Let $(Y,\textsl{g}_Y)$ be a complete Riemannian manifold with the isometric $H$-action, with a fixed point $y_0 \in Y$. Suppose that an $H$-invariant $C^1$ function $f : Y \to \rl$ is strictly convex along geodesics emanating from $y_0$ and not contained in the $H$-orbit.
	
	Then $f$ admits a critical point if and only if 
	\begin{equation*}
		\lim_{\tau \to + \infty} \frac{f(\gamma(\tau))}{\tau} >0
	\end{equation*}
	for any geodesic $\gamma$ emanating from $y_0$ that is not contained in the $H$-orbit. Moreover, any critical point must be the global minimiser of $f$ over $Y$ which is unique modulo the $H$-action.
\end{corollary}

\medskip

\noindent \textbf{Acknowledgements} The author thanks the anonymous referee for helpful suggestions. This work is partially supported by JSPS KAKENHI Grant Number 19K14524.

\end{document}